\documentclass[12pt]{article}
\usepackage{amsfonts}
\usepackage{graphicx,amssymb,amsfonts,amsmath,amscd}
\usepackage[all,ps,cmtip]{xy}
\usepackage{amscd}

\usepackage{mathrsfs}
\let\mathcal\mathscr

\usepackage{hyperref}

\def\Alb{\mathop{\rm Alb}\nolimits}

\def\codim{\mathop{\rm codim}\nolimits}
\def\dim{\mathop{\rm dim}\nolimits}

\def\llra{\hbox to 12mm{\rightarrowfill}}

\def\Pic{\mathop{\rm Pic}\nolimits}

\def\t{{\tau}}

\def\cF{{\mathcal F}}

\def\cJ{{\mathcal J}}

\def\cO{{\mathcal O}}

\def\1Y{{Y^{'}}}

\def\AA{{\widetilde{A/K}}}

\def\XX{{\widehat{X}}}
\def\VV{{\widetilde{V}}}
\def\WW{{\widehat{W}}}
\def\KK{{\widetilde{K}}}
\def\hh{{\widehat{h}}}

\def\ab{{\overline{b}}}

\newenvironment{proof}{\trivlist \item[\hskip\labelsep{\sc
Proof.}]\rm}{\hbox to.1pt{\hss}\hfill$\square$\bigskip\endtrivlist}

\newenvironment{proof5.1}{\trivlist \item[\hskip\labelsep{\sc
Proof of Theorem 5.1.}]\rm}{\hbox
to.1pt{\hss}\hfill$\square$\bigskip\endtrivlist}

\newtheorem{theo}{Theorem}[section]

\newtheorem{prop}[theo]{Proposition}
\newtheorem{lemm}[theo]{Lemma}
\newtheorem{coro}[theo]{Corollary}

\newtheorem{defi}[theo]{Definition}
\newtheorem{rema}[theo]{Remark}

\begin{document}
\title{An effective version of a theorem of Kawamata on the Albanese map}
\date{\today}
\author{Zhi Jiang}
\maketitle
To any  smooth complex projective variety $X$ are associated an abelian variety $\Alb(X)$ of dimension $q(X):=h^1(X,\cO_X)$, its {\em Albanese variety,} and a morphism $a_X:X\to \Alb(X)$, the {\em Albanese map,} which are very useful tools to study the geometry of $X$.

 Kawamata proved in \cite{KA}  that when the Kodaira dimension   $\kappa(X)$ is zero, the Albanese map is an algebraic fiber space, which means that:
\begin{itemize}
\item $a_X$ is surjective;
\item the fibers of $a_X$ are connected.
\end{itemize}
This kind of result (especially the second part) yields for example birational characterizations of abelian varieties:  $X$ is birational to an abelian variety if and only if $\kappa(X)=0$ and $q(X)=\dim(X)$.

However, the vanishing of  $\kappa(X)$ is not an effective condition (it means that the plurigenera $P_m(X):=h^0(X,\omega_X^m)$ are all 0 or 1 when $m> 0$ and that one of them is 1).
It is therefore
natural to try to prove the same result with weaker and effective
assumptions on the plurigenera of $X$.

For the surjectivity of $a_X$, this was done in a series of articles initiated by Koll\'ar (\cite{K}),   followed by
  Ein and Lazarsfeld (\cite{EL}) and later by Hacon and Pardini  (\cite{HAC1}) and Chen and Hacon (\cite{CH4}), who proved that $a_X$ is surjective if
$0< P_m(X)\leq 2m-3$ for some $m\ge2$, or if
 $P_3(X)=4$. We put here the finishing touch to this series by proving the following optimal result (Theorem \ref{hp}).

\medskip\noindent{\bf Theorem } {\em Let $X$ be a smooth complex projective variety. If
$$0<P_m(X)\leq 2m-2$$ for some $m\geq 2$, the Albanese map $a_X: X\to
\Alb(X)$ is surjective.} \medskip

When $C$ is a smooth projective curve of genus $2$, we have
$P_m(C)=2m-1$ for $m\geq 2$. However $a_C: C\rightarrow \Alb(C)$ is
not surjective. This example shows that without other assumptions,
our bound is optimal.

\medskip

As far as connectedness of the fibers of the Albanese map is concerned, they were no previous results in that direction. The main purpose of this paper is to show that there exists a
similar effective criterion for the Albanese morphism to be an algebraic fiber
space. More pecisely, we prove the following optimal bound (Theorem \ref{1} and Theorem \ref{10}).

\medskip\noindent{\bf Theorem } {\em Let $X$ be a smooth complex  projective variety. If $P_1(X)=P_2(X)=1$, or if
$$0<P_m(X)\leq m-2$$ for some $m\geq 3$, the Albanese map $a_X: X\rightarrow
\Alb(X)$ is an algebraic fiber space.} \medskip

Hacon and Pardini show   in \cite{HAC1} that for varieties with $P_3(X)=2$
and $q(X)=\dim(X)$, the Albanese map $a_X:
X\rightarrow \Alb(X)$ is a double covering. Hence $a_X$ is
surjective but does not have connected fibers. Furthermore,
$P_m(X)=m-1$ for any odd $m\geq 3$. From this example, we see that our
  result is optimal to a large extent.

\medskip
 As mentioned above, this criterion yields a numerical birational characterization of abelian varieties by adding $q(X)=\dim(X)$ to its hypotheses. The results and constructions developed here also lead to explicit descriptions of varieties with $q(X)=\dim (X)$ and small plurigenera, in the line of the series of papers \cite{CH1},
\cite{CH4}, \cite{HAC1}, and \cite{Ha}. For example, we can get a complete description of varieties with $P_2(X)=2$ and $q(X)=\dim (X)$.
We will come back to this in a future article.

\section{Preliminaries}
In this section we recall several theorems which will be used later. Throughout this article, we work over the fied of complex numbers and we denote  numerical equivalence by $\equiv$.

\medskip\noindent{\bf Vanishing theorem.} We state a result of Koll\'{a}r (\cite{K}, 10.15), which was
generalized later by Esnault and Viehweg.

\begin{theo}[Koll\'{a}r, Esnault-Viehweg]\label{kol}
Let $f: X\rightarrow Y$ be a surjective morphism from a smooth
projective variety $X$ to a normal variety $Y$. Let $L$ be a line
bundle on $X$ such that $L\equiv f^*M+\Delta$, where $M$ is a
$\mathbb{Q}$-Cartier $\mathbb{Q}$-divisor on $Y$ and $(X, \Delta)$
is klt. Then,
\begin{enumerate}
\item[a)] $R^jf_*(\omega_X\otimes L)$ is torsion free for $j\geq 0$;
\item[b)] if in addition, $M$ is big and nef, $H^i(Y, R^jf_*(\omega_X\otimes
L))=0$ for all $i>0$ and all $j\geq 0$.
\end{enumerate}
\end{theo}

\medskip\noindent{\bf Cohomological support loci.} These were first
studied by Green and Lazarsfeld for  the canonical bundle in \cite{GL1} and \cite{GL2},
through their generic vanishing theorems. Simpson
 also contributed to the subject (\cite{S}).

Let $X$ be a smooth projective variety and let $\mathcal {F}$ be a
coherent sheaf on $X$. The cohomological support loci of $\cF$ are
defined as $$V_i(X, \cF)=\{P\in \Pic^0(X)\mid H^i(X, \cF\otimes
P)\neq 0\},$$ which we often write as $V_i(\cF)$.

\medskip\noindent{\bf GV-objects.} These were
first considered by Hacon in \cite{HAC3} and systematically studied
by Pareschi and Popa in \cite{PP}. In this paper, we just need to
consider GV-sheaves with respect to the universal Poincar\'{e} line
bundle.
\begin{defi}A sheaf $\cF$ on $X$ is called a GV-sheaf if $$\codim_{\Pic^0(X)}V_i(\cF)\geq
i$$ for all $i\geq 0$.
\end{defi}
Let $a_X: X\rightarrow A$ be the Albanese map of $X$; then
$\Pic^0(X)$ is isomorphic to the dual abelian variety $\widehat{A}$. Let $M$ be an ample line bundle on
$\widehat{A}$. We denote by $\widehat{M}$ its Fourier-Mukai
transform, which is a locally free sheaf on $A$ (see \cite{Mu}). Let $\phi_M:
\widehat{A}\rightarrow A$ be the standard isogeny induced by $M$;
then $\phi_M^*\widehat{M}^{\vee}\simeq H^0(M)\otimes M$. Consider
the cartesian diagram:
\begin{equation}\label{t1}\CD
  \widehat{X} @>\varphi_M>> X \\
  @V a_{\widehat{X}} VV @V a_X VV  \\
  \widehat{A} @>\phi_M>> A
\endCD
\end{equation}
Hacon proved the following theorem in \cite{HAC3} (it was later
generalized by Pareschi and Popa in \cite{PP} Theorem A):
\begin{theo}\label{PP}Let $\cF$ be a coherent sheaf on a smooth projective variety $X$. If $H^i(\widehat{X},
\varphi_M^*\cF\otimes a_{\widehat{X}}^*M)=0$, for all $i> 0$ and any
sufficiently ample $M$, then $\cF$ is a GV-sheaf.
\end{theo}

Finally, the following elementary lemma from \cite{HAC1} will frequently be used.
\begin{lemm}\label{sim}Let $X$ be a smooth projective variety, let $L$ and $M$
be line bundles on $X$, and let $T\subset \Pic^0(X)$ be a subvariety of dimension $t$. If for some positive
integers $a$ and $b$ and all $P\in T$, we have $h^0(X, L\otimes
P)\geq a$ and $h^0(X, M\otimes P^{-1})\geq b$, then $h^0(X, L\otimes
M)\geq a+b+t-1.$
\end{lemm}
\section{When is the Albanese map surjective?}\label{s2}
 In this section I use the language of asymptotic multiplier ideal
sheaves. However many of the ideas come from \cite{K}, \cite{HAC1},
and \cite{HAC3}.
\begin{lemm}\label{mul}Suppose that $f: X\rightarrow Y$ is a surjective
morphism between smooth projective varieties, $L$ is a
$\mathbb{Q}$-divisor on $X$, and the Iitaka model of $(X, L)$
dominates $Y$. Assume that $D$ is a nef $\mathbb{Q}$-divisor on $Y$
such that $L+f^*D$ is a divisor on $X$. Then we have$$H^i(Y,
R^jf_*(\cO_X(K_X+L+f^*D)\otimes \cJ(||L||)\otimes Q))=0,$$ for all
$i\geq 1$, $j\geq 0$, and all $Q\in \Pic^0(X)$.
\end{lemm}
\begin{proof}
Let $m> 0$ be such that $mL$ is a divisor and
$\cJ(||L||)=\cJ(\frac{1}{m}|mL|)$ (\cite{Laz}, \S 11.2). Let $H$ be
a very ample divisor on $Y$. By assumption there exists an integer
$t>0$ such that $|tmL - f^*H|$ is non-empty.  Let $\mu:
X^{'}\rightarrow X$ be a log resolution such that:
\begin{eqnarray*}
\mu^*|tm L| &=& |L_1|+\sum_ia_iF_i,\\
\mu^*|tm L -f^*H| &=& |L_2|+\sum_ib_iF_i,
\end{eqnarray*}
where $|L_1|$ and $|L_2|$ are base-point-free, $\sum_ia_iF_i$ and
$\sum_ib_iF_i$ are the fixed divisors, and $\sum_iF_i +
\textmd{Exc}(\mu)$ is a divisor with simple normal crossings (SNC) support. Since
$\cJ(||L||)=\cJ(\frac{1}{m}|m L|)$, we also have
$\cJ(||L||)=\cJ(\frac{1}{tm}|tmL|)$, hence
$$\cJ(||L||)=\mu_*\cO_{X^{'}}\Big(K_{X^{'}/X}-\left\lfloor\frac{\sum_ia_iF_i}{tm}\right\rfloor\Big).$$
 Take
\begin{eqnarray*} B_1 &=& D_1+\sum_ia_iF_i\in \mu^*|tm L|\\
B_2 &=& D_2+\sum_ib_iF_i\in \mu^*|tm L -f^*H|
\end{eqnarray*}
where $D_1\in |L_1|$ and $D_2\in |L_2|$ are general elements, so
that $B_1+B_2$ is a divisor with SNC support. We then show that for
$k>0$ large enough, \begin{equation}\label{2}\left\lfloor
\frac{kB_1+B_2}{(k+1)tm}\right\rfloor=\left\lfloor\frac{\sum_ia_iF_i}{tm}\right\rfloor.\end{equation}
It is obvious that
$\left\lfloor\displaystyle\frac{kB_1+B_2}{(k+1)tm}\right\rfloor=\left\lfloor
\displaystyle\frac{\sum_i(ka_i+b_i)F_i}{(k+1)tm}\right\rfloor$. We write
$\frac{a_i}{tm}=m_i+s_i$ with
$m_i=\left\lfloor\frac{a_i}{tm}\right\rfloor$. Then,
$$\left\lfloor\frac{\sum_ia_iF_i}{tm}\right\rfloor=\sum_im_iF_i.$$ Because $H$ is
very ample on $Y$, we have $b_i\geq a_i$. Write $b_i=a_i+c_i$, with
$c_i\geq 0$. Then,
$$\left\lfloor\sum_i\frac{(ka_i+b_i)}{(k+1)tm}F_i\right\rfloor=\left\lfloor\sum_i\frac{((k+1)a_i+c_i)}{(k+1)tm}F_i\right\rfloor=\left\lfloor\sum_i(m_i+s_i+\frac{c_i}{(k+1)tm})F_i\right\rfloor.$$
Since $0\leq s_i <1$, we can let $k\geq 0$ be large enough such that
$s_i+\frac{c_i}{(k+1)tm}<1$, and this implies (\ref{2}). Then by local
vanishing (\cite{Laz}, Theorem 9.4.1),
\begin{eqnarray}\label{3}&&R^jf_{*}(\cO_X(K_X+ L+f^*D)\otimes
\cJ(||L||)\otimes
Q)\nonumber\\&=&R^j(f\circ\mu)_*\big(\cO_{X^{'}}\big(K_{X^{'}}+
\mu^*L+\mu^*f^*D-\left\lfloor\frac{kB_1+B_2}{(k+1)tm}\right\rfloor+\mu^*Q\big)\big),\end{eqnarray}
for all $j\geq 0$. We also have
\begin{eqnarray*}&
&\mu^*L+\mu^*f^*D-\left\lfloor\frac{kB_1+B_2}{(k+1)tm}\right\rfloor+\mu^*Q\\
&\equiv&
\mu^*L+\mu^*f^*D-\mu^*\frac{kL}{k+1}-\mu^*\frac{L}{k+1}+\mu^*f^*\frac{H}{(k+1)tm}+\left\{\frac{kB_1+B_2}{(k+1)tm}\right\}\\
&\equiv&
\mu^*f^*\frac{H}{(k+1)tm}+\mu^*f^*D+\left\{\frac{kB_1+B_2}{(k+1)tm}\right\}.\end{eqnarray*}
So Theorem \ref{kol} gives us that $$H^i\big(Y,
R^j(f\circ\mu)_* \cO_{X^{'}} (K_{X^{'}}+\mu^*L+\mu^*f^*L-\left\lfloor\frac{kB_1+B_2}{(k+1)tm}\right\rfloor+\mu^*Q  )\big)=0,$$
for all $i\geq 1$, all $j\geq 0$, and all $Q\in \Pic^0(X)$. By (\ref{3}),
this proves the lemma.
\end{proof}
The following lemma is essentially Proposition 2.12 in \cite{HAC1}.
I use Lemma \ref{mul} to make the proof a little bit simpler.
\begin{lemm}\label{ch}Let
$f:X\rightarrow Y$ be a surjective morphism between smooth
projective varieties and assume that the Iitaka model of $X$
dominates $Y$. Fix a torsion element $Q\in \Pic^0(X)$ and an integer $m\geq 2$. Then
$h^0(X, \omega_X^{m}\otimes Q\otimes f^*P)$ is constant for all
$P\in \Pic^0(Y)$.
\end{lemm}
\begin{proof}We consider $h^0(X,
\omega_X^{m}\otimes Q\otimes f^*P)$ as a function of $P\in
\Pic^0(Y)$. Let $P_0\in \Pic^0(Y)$ be such that $h^0(X,
\omega_X^{m}\otimes Q\otimes f^*P_0)=h$ is maximal. We are going to
prove that $$h^0(X, \omega_X^{m}\otimes Q\otimes f^*P_0\otimes
f^*P)=h,$$ for any torsion $P\in \Pic^0(Y)$. Since
$P_0+\{\textrm{torsion points}\}$ is dense in $\Pic^0(Y)$, we then
deduce the lemma from semicontinuity.

Let $P_1$, $P_2$, and $Q_1$ be such that $P_1^m=P_0$, $P_2^m=P$ and
$Q_1^m=Q$. From the properties of asymptotic multiplier ideal
sheaves (\cite{Laz}, Theorem 11.1.8), we know that
\begin{eqnarray*}& & H^0(X, \omega_X^{m}\otimes Q\otimes
f^*P_0\otimes f^*P)\\&=&H^0\big(X, \omega_X^{m}\otimes Q\otimes
f^*P_0\otimes f^*P\otimes \cJ(||\omega_X^{m}\otimes Q_1^{m}\otimes f^*P_1^{m}\otimes f^*P_2^{m}||)\big)\\
&=&H^0\big(X, \omega_X^{m}\otimes Q\otimes f^*P_0\otimes f^*P\otimes
\cJ(||\omega_X^{m-1}\otimes Q_1^{m-1}\otimes f^*P_1^{m-1}\otimes
f^*P_2^{m-1}||)\big).
\end{eqnarray*}
 Since
$P$ is a torsion point, there exists $N>0$ such that $P^N=\cO_Y$.
For $k>0$ large enough and divisible, we have
\begin{eqnarray*}& &\cJ(||\omega_X^{m-1}\otimes Q_1^{m-1}\otimes
f^*P_1^{m-1}\otimes f^*P_2^{i}||)\\&=&\cJ(
\frac{1}{kN}|(\omega_X^{m-1}\otimes Q_1^{m-1}\otimes
f^*P_1^{m-1}\otimes f^*P_2^{i})^{kN}|)\\&=&\cJ(
||\omega_X^{m-1}\otimes Q_1^{m-1}\otimes
f^*P_1^{m-1}||),\end{eqnarray*} for all $i\geq 0$. Hence we have
\begin{eqnarray*}
&&H^0(X, \omega_X^{m}\otimes Q\otimes f^*P_0\otimes
f^*P)\\&=&H^0\big(X, \omega_X^{m}\otimes Q\otimes f^*P_0\otimes
f^*P\otimes \cJ(||\omega_X^{m-1}\otimes Q_1^{m-1}\otimes
f^*P_1^{m-1}||)\big)\\&=&H^0\big(Y, f_*\big(\omega_X^{m}\otimes
Q_1^{m-1}\otimes f^*P_1^{m-1}\otimes \cJ(||\omega_X^{m-1}\otimes
Q_1^{m-1}\otimes f^*P_1^{m-1}||)\\&&\hskip 2cm{}\otimes Q_1\otimes f^*P_1\otimes
f^*P\big)\big).
\end{eqnarray*} We then apply Lemma \ref{mul} (the Iitaka model of
$(X, \omega_X^{m-1}\otimes Q_{1}^{m-1}\otimes f^*P_1^{m-1})$ dominates $Y$ by
assumption) to get that
$$h^0(X, \omega_X^{m}\otimes Q\otimes f^*P_0\otimes f^*P)= \chi\big(Y,
f_*\big(\omega_X^{m}\otimes Q\otimes \cJ(||\omega_X^{m-1}\otimes
Q_1^{m-1}\otimes f^*P_1^{m-1}||)\big)\big)$$ is the constant $h$.
\end{proof}



\begin{lemm}\label{har}Suppose that $f: X\rightarrow Z$ is an
algebraic fiber space between smooth projective varieties. Assume
that $P_m(X)\neq 0$, for some $m\geq 2$, that $H$ is a big
$\mathbb{Q}$-divisor on $Z$, and that $K$ is a nef
$\mathbb{Q}$-divisor on $Z$ such that $H_1\equiv H+K$ is a big and
nef divisor. Then,
\begin{itemize}
\item[1)]we have \begin{multline*}H^i\big(Z, R^jf_*\big(\cO_X(K_X+(m-1)K_{X/Z}+f^*H_1)\\\otimes \cJ(
||(m-1)K_{X/Z}+f^*H||)\big)\otimes P\big)=0,\end{multline*} for all $i\geq 1$,
$j\geq 0$ and all $P\in \Pic^0(Z)$.
\item[2)] the sheaf
$$f_*\big(\cO_X(K_X+(m-1)K_{X/Z})\otimes\cJ(||(m-1)K_{X/Z}+f^*H||)\big)$$ has rank $P_m(X_z)$, where $X_z$ is a general fiber of
$f$.
\end{itemize}
\end{lemm}

\begin{proof}The point here is the weak positivity of
$f_*(\omega_{X/Z}^{m-1})$, due to Viehweg (\cite{V1} Theorem 4.1 and
Corollary 7.1, or \cite{K} Proposition 10.2). There are two
conclusions:
\begin{itemize}\item[A.] the Iitaka model of $(X,
(m-1)K_{X/Z}+f^*H)$ dominates $Z$ and
\item[B.] there exists $k>0$ sufficient big and divisible such that
the restriction:
$$H^0(X, \cO_X(km(m-1)K_{X/Z}+kmf^*H))\rightarrow H^0(X_z, \cO_{X_z}(km(m-1)K_{X_z}))$$
is surjective, where $z\in Z$ is a general point.
\end{itemize}
By A, we can directly apply Lemma \ref{mul} to deduce item 1) in the lemma.

We take a log resolution $\t: X^{'}\rightarrow X$ such that the
restriction $\t_z: X^{'}_z\rightarrow X_z$ is also a log resolution
for sufficiently general $z\in Z$ (see \cite{Laz}, Theorem 9.5.35)
and fix such a point $z\in Z$. Set
\begin{itemize}\item $\t^*|km(m-1)K_{X/Z}+kmf^*H|=|L_1|+E_1$,
\item $\t_z^*|mK_{X_z}|=|L_2|+E_2,$
\end{itemize}
where $|L_1|$ and $|L_2|$ are base-point-free, $E_1$ and $E_2$ are
the fixed divisors, and $E_1+\textmd{Exc}(\t)$ has SNC support. We
have \begin{equation}\label{(3)}E_1|_{X_{z}^{'}}\preceq
k(m-1)E_2
\end{equation}
 by  B. Let $f^{'}:
X^{'}\xrightarrow{\t} X\xrightarrow{f} Z$ be the composition of
morphisms. Then $f^{'}$ is flat over a dense Zariski open subset of
$Z$. Hence the sheaf
$$f^{'}_{*}\big(\cO_{X^{'}}(K_{X^{'}}+(m-1)\t^*K_{X/Z}-\left\lfloor\frac{E_1}{km}\right\rfloor)\big)$$
has rank $$h^0\big(X_z^{'},
\cO_{X_z^{'}}\big(mK_{X^{'}_z}-\left\lfloor\frac{E_1}{km}\right\rfloor|_{X_{z}^{'}}\big)\big)=P_m(X_z).$$\\
We have the following inclusions
\begin{eqnarray*}
&
&f_*\t_*\cO_{X^{'}}\Big(K_{X^{'}}+(m-1)\t^*K_{X/Z}-\left\lfloor\frac{E_1}{km}\right\rfloor\Big)\\&\subset&
f_*\big(\cO_X(K_X+(m-1)K_{X/Z})\otimes\cJ(||(m-1)K_{X/Z}+f^*H||)\big)\\&\subset&
f_*(\cO_X(mK_{X}))\otimes\cO_Z(-(m-1)K_Z).
\end{eqnarray*}
Since the latter sheaf has rank $P_m(X_z)$, the middle sheaf
$f_*\big(\cO_X(K_X+(m-1)K_{X/Z})\otimes\cJ(||(m-1)K_{X/Z}+f^*H||)\big)$
also has rank $P_m(X_z)$.
\end{proof}

Under the assumptions of Lemma \ref{har}, we fix a big and
base-point-free divisor $H$. For $n>0$, we set
\begin{eqnarray*}
\cJ_{m-1,n}&=&\cJ(||(m-1)K_{X/Z}+\frac{1}{n}f^*H||)\\
\cF_{m-1, n}&=&f_*\big(\cO_X(K_X+(m-1)K_{X/Z})\otimes \cJ_{m-1,n}).
\end{eqnarray*}
By Lemma \ref{har}, $\cF_{m-1,n}$ has rank $P_m(X_z)>0$. These
sheaves were first considered by Hacon in \cite{HAC3}.

\begin{lemm} We have $\cJ_{m-1,n}\supset \cJ_{m-1,n+1}$ and
there exists $N>0$ such that for any $n\geq N$, one has
$\cF_{m-1,n}=\cF_{m-1,N}.$ We will denote by $\cF_{m-1,H}$ the fixed
sheaf $\cF_{m-1,N}$.
\end{lemm}
\begin{proof}
We may suppose that $k>0$ is such that the linear series
$|k(n+1)n((m-1)K_{X/Z}+\frac{1}{n}f^*H)|$ and
$|k(n+1)n((m-1)K_{X/Z}+\frac{1}{n+1}f^*H)|$ compute $\cJ_{m-1,n}$
and $\cJ_{m-1,n+1}$, respectively. Let $\t: X^{'}\rightarrow X$ be a
log resolution for both linear series. We can write
\begin{eqnarray*}
\t^*|k(n+1)n(m-1)K_{X/Z}+k(n+1)f^*H|&=&|L_1|+E_1,\\
\t^*|k(n+1)n(m-1)K_{X/Z}+knf^*H|&=&|L_2|+E_2,
\end{eqnarray*}
where $L_1$ and $L_2$ are base-point-free and $E_1$ and $E_2$ are
fixed divisors. Since $H$ is base-point-free, we have $E_2\succeq
E_1$. By the definition of asymptotic multiplier
ideal sheaves, $\cJ_{m-1, n}\supset\cJ_{m-1,n+1}$.\\
Take $H_1$ very ample on $Z$ such that $H_1-H$ is a nef divisor.
Then by Lemma \ref{har}, we have $$H^i(Z,
f_*\big(\cO_X(K_X+(m-1)K_{X/Z})\otimes\cJ_{m-1,n}\big)\otimes
\cO_Z(H_1))=0,$$ for $i\geq 1$. Using Hacon's argument in the proof
of Proposition 5.1 in \cite{HAC3}, there exists $N>0$ such that for
$n\geq N$, the inclusion
\begin{eqnarray*}
&&f_*(\cO_X(K_X+(m-1)K_{X/Z})\otimes\cJ_{m-1,N})\otimes\cO_Z(H_1)\\&\supset&
f_*(\cO_X(K_X+(m-1)K_{X/Z})\otimes\cJ_{m-1,n})\otimes\cO_{Z}(H_1)\end{eqnarray*}
is an equality. This implies that the inclusion
$$f_*(\cO_X(K_X+(m-1)K_{X/Z})\otimes\cJ_{m-1,N})\supset f_*(\cO_X(K_X+(m-1)K_{X/Z})\otimes\cJ_{m-1,n})$$
is again an equality.
\end{proof}

\begin{lemm}\label{3.5}Under the above assumptions, namely $f:X\rightarrow
Z$ is an algebraic fiber space between smooth projective  varieties
and $P_m(X)\neq 0$ with $m\geq 2$, we suppose moreover that $Z$ is
of maximal Albanese dimension and that $H$ is a big and
base-point-free divisor on $Z$ pulled back from $\Alb(Z)$. Then
$\cF_{m-1, H}$ is a nonzero GV-sheaf.
\end{lemm}
\begin{proof}  We apply Theorem \ref{PP}. Let $M$ be any ample divisor on $\Pic^0(Z)$. We have cartesian
diagrams as in (\ref{t1}):
$$\CD
   \widehat{X} @> \upsilon_M>> X\\
  @V \widehat{f} VV @V f VV\\
  \widehat{Z} @>\varphi_M>> Z \\
  @V a_{\widehat{Z}} VV @V a_Z VV  \\
  \Pic^0(Z) @>\phi_M>> \Alb(Z)
\endCD$$
where horizontal maps are \'{e}tale. By Theorem 11.2.16 in
\cite{Laz}, for any $n>0$,
$$\upsilon_M^*\cJ(||(m-1)K_{X/Z}+\frac{1}{n}f^*H||)=\cJ(||(m-1)K_{\widehat{X}/\widehat{Z}}+\frac{1}{n}\widehat{f}^*\varphi_M^*H||),$$
hence by flat base change
\begin{eqnarray*}
&&\varphi_M^*f_*\big(\cO_X(K_X+(m-1)K_{X/Z})\otimes
\cJ(||(m-1)K_{X/Z}+\frac{1}{n}f^*H||)\big)\\&=&\widehat{f}_*\big(\cO_{\widehat{X}}(K_{\XX}+(m-1)K_{\XX/\widehat{Z}})\otimes
\cJ(||(m-1)K_{\widehat{X}/\widehat{Z}}+\frac{1}{n}\widehat{f}^*\varphi_M^*H||)\big).
\end{eqnarray*}
It follows that
$$\varphi_M^*\cF_{m-1,
H}=\widehat{f}_*\big(\cO_{\widehat{X}}(K_{\XX}+(m-1)K_{\XX/\widehat{Z}})\otimes
\cJ(||(m-1)K_{\widehat{X}/\widehat{Z}}+\frac{1}{n}\widehat{f}^*\varphi_M^*H||)\big)$$
for all $n\gg 0$. Since $H$ is a divisor pulled back by $a_Z$, we
can take $n$ such that $n a_{\widehat{Z}}^*M-\varphi_M^*H$ is nef.
Then Lemma \ref{har} gives us the vanishing of $$H^i(\widehat{Z},
\varphi_M^*\cF_{m-1,H}\otimes a_{\widehat{Z}}^*M),$$ for all $i>0$
and we are done.
\end{proof}

\begin{lemm}\label{3.6}In the situation of Lemma \ref{3.5}, denoting by $a_Z: Z\rightarrow A$ the Albanese morphism of
$Z$, we have $R^ja_{Z*}(\cF_{m-1,H})=0$, for all $j>0$. Hence
$$V_i(\cF_{m-1,H})=V_i(a_{Z*}(\cF_{m-1,H})),$$ for all $i\geq 0$.
\end{lemm}
\begin{proof}Suppose that $R^ta_{Z*}(\cF_{m-1,H})\neq 0$ for some $t>0$. Let $H_1$ be a ample divisor
on $A$ such that $$H^k(A,
R^ja_{Z*}(\cF_{m-1,H})\otimes\cO_{A}(H_1))=0$$ for all $k\geq 1$ and
$j\geq 0$ and
$$H^0(A, R^ta_{Z*}(\cF_{m-1,H})\otimes\cO_{A}(H_1))\neq 0.$$
By the Leray spectral sequence, we have $$H^t(Z, \cF_{m-1,H}\otimes
\cO_{Z}(a_Z^*H_1))\neq0.$$Since $H$ is pulled back from $A$, we may
take $H_1$ such that $a_Z^*H_1-H$ is big and nef, then by Lemma
\ref{har}, we have $H^t(Z, \cF_{m-1,H}\otimes\cO_Z(a_Z^*H_1))=0$,
which is a contradiction. Thus $R^ja_{Z*}(\cF_{m-1,H})=0$ for all
$j>0$. For any $P\in \Pic^0(Z)$, we have $H^i(Z, \cF_{m-1,H}\otimes
a_Z^*P)\simeq H^i(A, a_{Z*}(\cF_{m-1,H})\otimes P)$, hence
$V_i(\cF_{m-1,H})=V_i(a_{Z*}(\cF_{m-1,H}))$ for all $i\geq 0$.
\end{proof}
\begin{coro}\label{3.7}The cohomological support $V_0(\cF_{m-1,H})$ is not
empty.
\end{coro}
\begin{proof}By Lemma \ref{3.5}, $\cF_{m-1,H}$ is a GV-sheaf, hence
(\cite{HAC3}, Corollary 3.2)$$V_0(\cF_{m-1,H})\supset
V_1(\cF_{m-1,H})\supset\cdots\supset V_d(\cF_{m-1,H}).$$ If
$V_0(\cF_{m-1,H})$ is empty, $V_i(\cF_{m-1,H})$ is empty for all
$i\geq 0$, hence $$H^i(Z,\cF_{m-1,H}\otimes a_Z^*P)=H^i(A,
a_{Z*}\cF_{m-1,H}\otimes P)=0,$$for all $i\geq 0$. By the properties of the
Fourier-Mukai transform on an abelian variety (see \cite{Mu}),
$a_{Z*}\cF_{m-1,H}=0$. However this is impossible since $a_Z$ is
generically finite and $\cF_{m-1,H}$ is a sheaf with positive rank.
\end{proof}
\begin{theo}\label{hp}Let $X$ be a smooth projective variety. If
$$0<P_m(X)\leq 2m-2,$$ for some $m\geq 2$, the Albanese map $a_X: X\rightarrow
\Alb(X)$ is surjective.
\end{theo}
\begin{proof}If $a_X$ is not surjective,   by Ueno's theorem (\cite{Mo}, Theorem
(3.7)), upon replacing $X$ by a birational model, there exists a
surjective morphism $f_1: X\rightarrow Z_1$ onto a smooth variety
$Z_1$ of general type of dimension $d>0$ such that $Z_1\rightarrow
\Alb(Z_1)$ is a birational map onto its image and $Z_1\rightarrow
\mathbb{P}(H^0(Z_1, \cO_{Z_1}(K_{Z_1})))$ is a map generically finite onto its
image. Obviously, $P_k(Z_1)\geq \binom{d+k}{d}$ for all $k\geq 1$.
Taking the Stein factorization and making birational modifications,
we may suppose that there is an algebraic fiber space $f:
X\rightarrow Z$ such that $Z$ is a smooth variety of general type
and of maximal Albanese dimension $d$, and $P_k(Z)\geq \binom{d+k}{k}$
for all $k\geq 1$.

We let $H$ be a big and base-point-free divisor pulled back by the
Albanese morphism $a_Z: Z\rightarrow \Alb(Z)$. By Corollary
\ref{3.7}, $V_0(\cF_{m-1,H})$ is not empty thus there exists $P\in
\Pic^0(Z)$ such that $h^0(Z, \cF_{m-1,H}\otimes P)\geq 1$. Hence
\begin{equation}\label{nnn}
h^0(X, \cO_X(K_X+(m-1)K_{X/Z})\otimes f^*P)\geq 1.
\end{equation}
On the other
hand, we have $h^0(X, \cO_X((m-1)f^*K_Z))\geq \binom{d+m-1}{m-1}$.
We get
\begin{equation}\label{5}h^0(X, \cO_X(mK_X)\otimes f^*P)\geq
\binom{d+m-1}{m-1}.\end{equation}Since $Z$ is of general type, the
Iitaka model of $(X,K_X)$ dominates $Z$ because of (\ref{nnn}), hence we apply Lemma
\ref{ch} to get $h^0(X,\cO_X(mK_X))\geq \binom{d+m-1}{m-1}$.

If $\dim (Z)=d\geq 2$, then $P_m(X)\geq \binom{m+1}{2}\geq 2m-1$,
which is a contradiction.

If $\dim (Z)=1$, $P_m(X)=h^0(Z,f_*(\omega_{X/Z}^m )\otimes
\omega_Z^{m})$. As in Corollary 3.6 in \cite{V},
$f_*(\omega_{X/Z}^m)$ is a nonzero nef vector bundle on $Z$ hence has nonnegative degree. By the
Riemann-Roch theorem, we obtain $P_m(X)\geq 2m-1$, again a
contradiction.
\end{proof}
\begin{rema}\upshape
The proof follows ideas of Koll\'{a}r's (\cite{K}), later improved
by Hacon and Pardini. Briefly speaking, Koll\'{a}r proved that
$P_m(X)\geq P_{m-2}(Z)$ and Hacon and Pardini used the finite map
$$|(m-2)K_Z+P|\times |K_X+(m-1)K_{X/Z}+K_Z-f^*P|\rightarrow
|mK_X|,$$ where $P\in \Pic^0(Z)$, to give a better estimate of
$P_m(X)$. However, the dimension  $h^0(Z, \cO_Z(kK_Z))$ grows very fast with $k$, so
my starting point was to prove $P_m(X)\geq P_{m-1}(Z)$ by applying
the theory of GV-sheaves.
\end{rema}

\begin{coro}Suppose that $0<P_m(X)<\binom{d+m}{m-1}$ for some $m\geq 2$ and $d\geq 1$.  Then
$\kappa(a_X(X))\leq d$.
\end{coro}

\begin{proof}It is just (\ref{5}) in the proof of Theorem \ref{hp}, where by Ueno's theorem
$d$ is the Kodaira dimension of $a_X(X)$.
\end{proof}

\section{When does the Albanese map have connected fibers?}\label{s3}
Ein and Lazarsfeld in \cite{EL} gave another proof of
Kawamata's theorem based on the generic vanishing theorem. Their proof is actually very close to an effective result. With the help of a
proposition of Chen and Hacon, we prove the following:

\begin{theo}\label{1}Let $X$ be a smooth projective variety with
$P_1(X)=P_2(X)=1$. The Albanese map $a_X: X\rightarrow \Alb(X)$ is
an algebraic fiber space.
\end{theo}

\begin{proof}Let $A$ be the Albanese variety of $X$. The Albanese morphism is already surjective by \cite{HAC1}. Suppose that it has non-connected fibers. We start with the Stein
factorization of $a_X$ and, resolving singularities and
indeterminacies, we can assume that $a_X$ admits a factorization
$$
X\xrightarrow{g}V\xrightarrow{b}A,$$
where $b$ is a generically
finite non birational morphism, $g$ is surjective with connected
fibers, $V$ is smooth and projective. Since $a_X$ is the Albanese
morphism of $X$, $V$ is not birational to an abelian variety. Thus
$V$ is of maximal Albanese dimension and by Chen and Hacon's
characterization of abelian varieties (\cite{CH1}, Theorem 3.2), we
have $P_2(V)\geq 2$. We set $\dim (X)=n$ and $\dim (V)=\dim (A)=d$.

Since $P_1(X)=P_2(X)=1$, $0\in V_0(X, \omega_X)$ is an isolated
point (\cite{EL}, Proposition 2.1). Hence $0\in V_0(V, g_*\omega_X)$
is also an isolated point. By Proposition 2.5 in \cite{CH3}, for any
$v\neq0$ in $H^1(V, \cO_V)$, the sequence
$$0\rightarrow H^0(V, g_{*}\omega_X)\xrightarrow{\cup v}H^1(V, g_{*}\omega_X)\rightarrow\cdots \xrightarrow{\cup v}H^d(V, g_{*}\omega_X)\rightarrow 0$$
is exact. Since $b$ is surjective,  we may, through the map $b^*$,
consider $H^1(A, \cO_A)$ as a subspace of $H^1(V, \cO_V)$. Then, as
in the proof of Theorem 3 in \cite{EL}, we have an exact complex of
vector bundles on $\mathbf{P}=\mathbf{P}(H^1(A,
\cO_A))=\mathbf{P}^{d-1}$:
\begin{multline*}
0\rightarrow H^0(V, g_{*}\omega_X)\otimes
\cO_{\mathbf{P}}(-d)\rightarrow H^1(V,
g_{*}\omega_X)\otimes\cO_{\mathbf{P}}(-d+1)\to\cdots\\
\cdots
\rightarrow H^d(V, g_{*}\omega_X)\otimes\cO_{\mathbf{P}} \rightarrow
0.
\end{multline*}
Take $(v_1,\ldots,v_d)$ a basis for $H^1(A, \cO_A)$. By chasing
through the diagram, we obtain that $H^0(V,
g_{*}\omega_X)\xrightarrow{\wedge v_1\wedge\cdots \wedge v_d}H^d(V,
g_{*}\omega_X)$ is an isomorphism.

By Theorem 3.4 in \cite{K4}, $$H^d(X, \omega_X)\simeq
\bigoplus_iH^i(V, R^{d-i}g_*\omega_X).$$ Hence we have
$$\xymatrix@M=5pt{
H^0(V, g_*\omega_X)\ar[rr]_{\simeq}^{\wedge v_1\wedge\cdots
\wedge v_d}\ar[d]^{\simeq}&&H^d(V, g_{*}\omega_X)\ar@{^{(}->}[d]\\
H^0(X, \omega_X)\ar[rr]^{\wedge g^*(v_1\wedge\cdots \wedge
v_d)}&&H^d(X, \omega_X)}$$ By Hodge conjugation and Serre duality
$H^d(X, \omega_X)\simeq H^0(X, \Omega^{n-d}_X)$. We will denote by
$E\subset H^0(X, \Omega_X^{n-d})$ the nonzero subspace corresponding
to $H^d(V, g_{*}\omega_X)\subset H^d(X, \omega_X)$. Let
$(\eta_1,\ldots,\eta_d)$ in $H^0(A, \Omega_A)$ be the conjugate
basis of $(v_1,\ldots,v_d)$. By Serre duality and Hodge conjugation,
we get from the above diagram that $$E\xrightarrow{\wedge
g^*(\eta_1\wedge\cdots\wedge \eta_d)}H^0(X, \omega_X)$$ is an
isomorphism. Since $\eta_1\wedge\cdots\wedge \eta_d$ is a nonzero
section of $K_V$, we have $K_X\succeq g^*K_V$. We deduce $P_2(X)\geq
P_2(V)\geq 2$, which is a contradiction.
\end{proof}

The proof of Theorem \ref{1} is closely related to Green and
Lazarsfeld's generic vanishing theorem, which is Hodge-theoretic.
Meanwhile Theorem \ref{hp} relies heavily on the weak positivity
theorem of Viehweg. It is natural to ask whether we can use the
ideas in section \ref{s2} to prove other criteria to tell when the Albanese
map is an algebraic fiber space.

We again let $A$ be $\Alb(X)$. Suppose that $a_X: X\rightarrow A$ is
surjective but has non-connected fibers. We take the Stein
factorization and obtain that $a_X$ factors as
$X\xrightarrow{g}V\xrightarrow{b}A$ where $V$ is normal and finite
over $A$ with, again $P_2(V)\geq 2$. The problem here is that we
cannot expect the image of the Iitaka fibration of $V$ to be of general type.

Fortunately, a structure theorem for varieties of maximal Albanese
dimension due to Kawamata (Theorem 13 in \cite{KA}) tells us that
the situation is still manageable.
\begin{theo}[Kawamata]\label{K2}Let $b: V\rightarrow A$ be a finite morphism
from a projective normal algebraic variety to an abelian variety.
Then $\kappa(V)\geq 0$ and there are an abelian subvariety $K$ of
$A$, \'{e}tale covers $\VV$ and $\KK$ of $V$ and $K$ respectively, a
projective normal variety $\WW$, and a finite abelian group $G$,
which acts on $\KK$ and faithfully on $\WW$, such that:
\begin{itemize}
\item[(1)] $\WW$ is finite over $A/K$, of general type and of dimension $\kappa(V)$,
\item[(2)] $\VV$ is isomorphic to $\KK\times
\WW$,
\item[(3)] $V=\VV/G=(\KK\times\WW)/G$, where $G$ acts diagonally and freely on $\VV$.
\end{itemize}
\end{theo}

The construction of $\WW$ and $\VV$ is crucial for our purpose so I
will recall the proof of this theorem following Kawamata.

Let $\delta: V^{'}\rightarrow V$ be a birational modification of $V$
such that $V^{'}$ is smooth and there exists a morphism $h^{'}:
V^{'}\rightarrow W^{'}$ such that $W^{'}$ is also smooth and $h^{'}$
is a model of the Iitaka fibration of $V$. Then a general fiber
$V^{'}_{w^{'}}$ of $h^{'}$ is smooth, of Kodaira dimension 0, and
generically finite over an abelian variety, hence by Kawamata's
theorem, $V^{'}_{w^{'}}$ is birational to an abelian
variety and $(b\circ \delta)(V^{'}_{w^{'}})$ is then an abelian
subvariety of $A$, denoted by $K_{w^{'}}$. Since $w^{'}$ moves
continuously, $K_{w^{'}}$ is a translate of a fixed abelian
subvariety $K\subset A$ for every $w^{'}\in W^{'}$. Let $\pi:
A\rightarrow A/K$ be the quotient map.

Consider the Stein factorization $$\pi\circ b:
V\xrightarrow{h}W\xrightarrow{b_W} A/K.$$ Since general fibers of
$h^{'}$ are contracted by $\pi\circ b\circ \delta$, hence by $h\circ
\delta$, the map $h\circ \delta$ factors through $h^{'}$ by
rigidity, and we get the following commutative diagram:
\begin{eqnarray}\label{d6}
\xymatrix@M=5pt@C=40pt
{ V^{'}\ar[r]^{\delta}\ar[d]^{h^{'}}& V \ar[r]^b_{\rm{finite}}\ar[d]^{h}& V_0\ar@{^{(}->}[r]\ar[d]&A\ar[d]^{\pi} \\
W^{'}\ar[r]^{\delta^{'}}& W \ar[r]^{b_W}_{\rm{finite}}&
W_0\ar@{^{(}->}[r]&A/K}
\end{eqnarray}
where $W$ is normal, $b_W$ is finite, $h: V\rightarrow  W$ has
connected fibers, $\delta$ and $\delta^{'}$ are birational, and $V_0$
and $W_0$ are the images of $V$ and $W$ in $A$ and $A/K$
respectively.

By Poincar\'{e} reducibility, there exists an isogeny
$\widetilde{A/K}\rightarrow A/K$ such that $A\times_{A/K}\AA\simeq
K\times \AA$. We then apply the \'{e}tale base change
$(\cdot)\times_{A/K}\AA\rightarrow \cdot$ in the diagram (\ref{d6}) and get
the following commutative diagram:
\begin{eqnarray*}\xymatrix@M=5pt@R=13pt
{
&\widetilde{V}\ar[rr]^-{\widetilde{b}}_-{\rm finite}\ar'[d][dd]_(.3){\widetilde{h}}\ar[dl]&&K\times
\widetilde{W_0}=\widetilde{V_0}\ar'[d][dd]\ar@<-.5ex>@{^{(}->}[rr]\ar@<-.5ex>[dl]&& K\times \AA\ar[dd]\ar[dl]\\
V\ar[dd]_{\begin{matrix}{\rm\scriptstyle fiber}\\{\rm \scriptstyle space}\end{matrix}}^{h}\ar[rr]^(.6){b}\ar[dd]&&V_0\ar@<-.2ex>@{^{(}->}[rr]\ar[dd]&&A\ar[dd]\\
& \widetilde{W}\ar'[r][rr]^(.4){\widetilde{b}_W}_(.4){\rm{finite}}\ar[dl]&& \widetilde{W_0}\ar[dl]\ar@{^{(}->}'[r][rr]&&\AA\ar[dl]\\
W\ar[rr]^{b_W}_{\rm{finite}}&& W_0\ar@<-.5ex>@{^{(}->}[rr]&&A/K}
\end{eqnarray*}
where $\widetilde{W_0}$ is some connected component of the inverse
image of $W_0$ in $\AA$, $\widetilde{V}$ is some connected component
of $V\times_{V_0}\widetilde{V_0}$,   $\widetilde{W}$ is some
connected component of $W\times_{W_0}\widetilde{W_0}$, and all
slanted arrows are \'{e}tale.

Let us look at
\begin{eqnarray*}
\xymatrix{
\widetilde{V}\ar[rr]^{\widetilde{b}}_{\rm{finite}}\ar[d]^{\widetilde{h}}&& K\times\widetilde{W_0}\ar[d]\\
\widetilde{W}\ar[rr]^{\widetilde{b}_W}_{\rm{finite}}&&
\widetilde{W_0}.}
\end{eqnarray*}
A general fiber of $\widetilde{h}$ is an \'{e}tale cover of a
general fiber of $h$ hence an \'{e}tale cover of $K$, thus
isomorphic to an abelian variety $\KK$.

The morphism $\widetilde{b}$ is \'{e}tale over a product $K\times
U_0$ for $U_0$ a dense Zariski open subset of $\widetilde{W_0}$:
$$
\xymatrix@C=40pt@M=5pt{
\widetilde{h}^{-1}(U)\ar[d]_{\rm{smooth}}\ar[r]^{\widetilde{b}} &K\times U_0\ar[d]\\
 U\ar[r]  &  U_0. }$$

The group $K$ acts on $\widetilde{V_0}=K\times \widetilde{W_0}$, and
on $K\times U_0$. The infinitesimal action corresponds to vector
fields, which lift to $\widetilde{b}^{-1}(K\times U_0)$ because
$\widetilde{b}$ is \'{e}tale there.

This induces an action of $\KK$ on
$\widetilde{h}^{-1}(U)=\widetilde{b}^{-1}(K\times U_0)$ hence a
rational action on $\widetilde{V}$. Let $\widetilde{k}\in \KK$ and let $k\in
K$ be its image. Let $\widetilde{\Gamma}\subset
\widetilde{V}\times\widetilde{V}$ and $\Gamma\subset
\widetilde{V_0}\times \widetilde{V_0}$ be the graphs of the actions
of $\widetilde{k}$ and $k$ respectively. We have
\begin{eqnarray}\xymatrix@M=5pt@C=40pt{
 \widetilde{V}\times \widetilde{V}\ar[d]^{(\widetilde{b},\widetilde{b})}
& \widetilde{\Gamma}\ar@<.5ex>@{_{(}->}[l]\ar[r]^{\widetilde{pr_1}}\ar[d]&\widetilde{V}\ar[d]^{\ab}\\
 \widetilde{V_0}\times \widetilde{V_0} & \Gamma\ar@<.5ex>@{_{(}->}[l]\ar[r]^{pr_1} &\widetilde{V_0},}
\end{eqnarray}
where $(\widetilde{b},\widetilde{b})$ is finite and $pr_1$ is an
isomorphism. We see that $\widetilde{pr_1}$ is finite and birational
hence an isomorphism because $\widetilde{V}$ is normal. Thus the
action of $\widetilde{k}$ is an isomorphism. So $\KK$ acts on $\widetilde{V}$
and $\widetilde{b}$ is equivariant for the $\KK$-action on
$\widetilde{V}$ and the $K$-action on $\widetilde{V_0}$.

Set $G_1=\KK/K$.
For $y\in \widetilde{W_0}$ general, we have
$$\widetilde{h}^{-1}\widetilde{b_W}^{-1}(y)=\widetilde{b}_W^{-1}(y)\times \KK=\widetilde{b}^{-1}(K\times\{y\}),$$
hence $$\deg\widetilde{b}=\sharp G_1\cdot \deg\widetilde{b}_W.$$

Set $\WW_0=\widetilde{b}^{-1}(k\times \widetilde{W_0})$ for $k\in K$
general. Then $\WW_0$ is normal and $G_1$ acts on $\WW_0$ ($\WW_0$
may be not connected).
We have a diagram:
\begin{eqnarray*}
\xymatrix{
 \WW_0 \ar[rr]^{\deg\widetilde{b}:1}\ar[d]^{\sharp G_1:1}&& k\times \widetilde{W_0}\ar@{=}[d]\\
\widetilde{W}\ar[rr]^{\deg\widetilde{b}_W:1} && \widetilde{W_0},}
\end{eqnarray*}
hence $\WW_0/G_1=\widetilde{W}$.

Note that $G_1$ acts on $\KK\times \WW_0$ diagonally and freely
(because the action is free on $\KK$). By the $\KK$-action, we have
a morphism $\varphi: \KK\times \WW_0\rightarrow \widetilde{V}$ and
there is a commutative diagram:
\begin{eqnarray*}
\xymatrix@C=40pt@M=5pt{
 \KK\times\WW_0 \ar[d]\ar[r]^(.6){\varphi}& \widetilde{V}\ar[d]^{\widetilde{h}}\\
 \WW_0\ar[r]^{\rm{finite}}& \widetilde{W}.}
\end{eqnarray*}
Thus $\varphi$ is finite because any contracted curve is in some
$\KK\times \widetilde{w}$ but because of the $\KK$-action, this is
impossible.

From the diagram, we have a finite morphism $\KK\times
\WW_0\rightarrow \widetilde{V}\times_{\widetilde{W}}\WW_0$. Since it
is birational over $U$, it is an isomorphism. Hence
$$\widetilde{V}=(\widetilde{V}\times_{\widetilde{W}}\WW_0)/G_1=(\KK\times
\WW_0)/G_1.$$ We then let $\WW$ be a connected component of $\WW_0$
and let $\widetilde{\VV}=\KK\times \WW$. Then $\widetilde{\VV}$ is
still a Galois \'{e}tale cover of $\widetilde{V}$. There exists a
commutative diagram:
\begin{eqnarray*}
\xymatrix@C=40pt@M=5pt@R=15pt{\widetilde{\VV}\ar[r]\ar[d]&\KK\times \widetilde{A/K}\ar[d]\\
\widetilde{V}\ar[r]\ar[d]&K\times \widetilde{A/K}\ar[d]\\
V\ar[r]&A.}
\end{eqnarray*}
We then conclude that $\widetilde{\VV}$ is a connected component of
$V\times_A (\KK\times\widetilde{A/K})$. Let $G_2$ be the finite
abelian group $(\KK\times\widetilde{A/K})/A$. Then
$V=\widetilde{\VV}/G=(\KK\times\WW)/G$, for some quotient group $G$
of $G_2$, where $G$ acts diagonally. Since any quotient of $\KK$ by
a subgroup of $G$ is still an abelian variety, we may assume that
$G$ acts faithfully on $\WW$.

A crucial fact is that $\WW$ is of general type because
$$\kappa(\WW)=\kappa(\widetilde{\VV})=\kappa(V)=\dim (W)=\dim (\WW).$$

We put everything in a commutative diagram:
\begin{eqnarray}\label{d7}
\xymatrix@R=30pt@C=35pt@M=+5pt{
  \widetilde{\VV}=\KK\times\WW\ar[rr]^-{\pi_{\widetilde{V}}}\ar[d]^{\hh=pr_2}\ar@/^2pc/[rrr]^{\rm{Galois\; \rm\acute etale}}&& \widetilde{V} \ar[r]^{\pi_V}\ar[d]^{\widetilde{h}}& V\ar[d]^{h}\ar[r]^{b}_{\rm{finite}}& A\ar[d]^{\pi} \\
 \WW \ar[rr]^{b_{\widetilde{W}}}_{\rm{Galois}}\ar@/_2pc/[rrr]_{b_{\WW}}&& \widetilde{W} \ar[r]^{\pi_W}_{\rm{finite}}& W \ar[r]^{b_W}_{\rm{finite}}& A/K.
}
\end{eqnarray}\\\\
We are now ready to prove the main theorem.

\begin{theo}\label{10}Let $X$ be a smooth projective variety. If $$0<P_m(X)\leq m-2,$$
for some $m\geq 3$, the Albanese map $a_X: X\rightarrow A$ is an
algebraic fiber space.
\end{theo}
\begin{proof}
By Theorem \ref{hp}, $a_X$ is already surjective. Suppose
that it has non-connected fibers. Again we have the Stein
factorization $a_X:
X\xrightarrow{g}V\xrightarrow{b}A$, where $g$ has connected fibers,
$V$ is normal, and $b$ is finite not birational. Applying the above
description of the structure of $V$ in (\ref{d6}) and (\ref{d7}), we get the
following commutative diagram:
\begin{eqnarray}\label{d8}\xymatrix@C=40pt{
X\times_{V}\widetilde{\VV}
\ar[r]^(.6){\pi_X} \ar[d]^{\widehat{g}} & X\ar[d]^{g}\ar[dr]^{a_X}\\
\widetilde{\VV}\ar[r]^{\rm{Galois}}_{\rm\acute etale}\ar[d]^{\widehat{h}} & V\ar[d]^h\ar[r]^b &A\ar[d]^{\pi}\\
\WW\ar[r]^{b_{\WW}} & W\ar[r]^{b_W}&A/K,}
\end{eqnarray}
where $\pi_X$ is \'{e}tale Galois with Galois  group $G$,
$\widetilde{\VV}=\WW\times \KK$, and $\WW$ is of general type.

There exists a dense Zariski open subset $U$ of $W$ such that $U$
and $b_{\WW}^{-1}(U)$ are smooth and $h\circ g$ and
$\widehat{h}\circ\widehat{g}$ are smooth over $U$ and
$b_{\WW}^{-1}(U)$ respectively. Through Hironaka's resolution of
singularities, we can blow up $W$ and $X$ along smooth subvarieties of
$W-U$ and $X-(h\circ g)^{-1}(U)$ respectively and assume that $W$ is
smooth. Similarly, let $W_1$ and $X_1$ be the smooth projective
varieties obtained by blowing-up $\WW$ and
$X\times_V\widetilde{\VV}$ along subvarieties of
$\WW-b_{\WW}^{-1}(U)$ and
$X\times_V\widetilde{\VV}-(b_{\WW}\circ\widehat{h}\circ\widehat{g})^{-1}(U)$
respectively such that we have the following commutative diagram:
\begin{eqnarray}\label{d9}
\xymatrix{ X_1\ar[dr]^{\epsilon}\ar[rr]^{\pi_{X_1}}\ar[dd]_{f_1} &&X\ar[dd]^f\\
& X\times_V\widetilde{\VV}\ar[ur]^{\pi_X}&\\
 W_1\ar[rr]^{b_{W_1}}&&W,}
\end{eqnarray}
where $W_1$ is of general type, $b_{W_1}$ is generically finite and
$\epsilon$ is the blow-up of $X\times_V\widetilde{\VV}$. We write
$$K_{X_1}=\pi_{X_1}^*K_X+E,$$ where $E$ is an effective exceptional divisor for
$\pi_{X_1}$, $f_1(E)$ is a subvariety of $W_1-b_{W_1}^{-1}(U)$,
and
\begin{eqnarray*}\pi_{X_1*}\cO_{X_1}=\pi_{X*}\epsilon_*\cO_{X_1}=\pi_{X*}\cO_{X\times_{V}\VV}=\bigoplus_{\chi\in G^*}P_{\chi},\end{eqnarray*} where
$P_{\chi}\in\Pic^0(X)$ is the torsion line bundle corresponding to
$\chi\in G^*$.

In order to prove the theorem, we will need to treat two
cases, $\kappa(W)>0$ or $\kappa(W)=0$. The strategies of the proofs
are the same so I will treat the first case in detail and explain how
  very similar arguments work for the second case.

\begin{lemm}\label{4.5}Let $X$ be a smooth projective variety with $P_m(X)>0$ for some $m\geq 2$.
Let $f: X\rightarrow W$ be as above. The Iitaka model of $(X,
(m-1)K_{X/W}+f^*K_W)$ dominates W.
\end{lemm}
\begin{proof}
We use the same notation as above. In (\ref{d9}), we already know that
$W_1$ is of general type so by Viehweg's result (see the proof of
Lemma \ref{har}), the Iitaka model of $(X_1,
(m-1)K_{X_1/W_1}+f_1^*K_{W_1})$ dominates $W_1$. On the other hand,
we can write
\begin{eqnarray}\label{d10}&&(m-1)K_{X_1/W_1}+f_1^*K_{W_1}\nonumber\\&=&\pi_{X_1}^*((m-1)K_{X/W}+f^*K_W)-(m-2)f_1^*K_{W_1/W}+(m-1)E.\end{eqnarray}
Since $K_{W_1/W}$ is effective, the Iitaka model of $(X_1,
\pi_{X_1}^*((m-1)K_{X/W}+f^*K_W)+(m-1)E)$ dominates $W_1$. Hence for
any ample divisor $H$ on $W$, there exists $N>0$ such that
$\pi_{X_1}^*\cO_{X}(N((m-1)K_{X/W}+f^*K_W)-f^*H)\otimes\cO_{X_1}(N(m-1)E)$ has a nonzero section. Since $\pi_{X_1*}\cO_{X_1}(N(m-1)E)=\pi_{X_1*}\cO_{X_1}$
is a direct sum of torsion line bundles, there exists $k>0$ such
that $kN((m-1)K_{X/W}+f^*K_W)-kf^*H$ is effective. Therefore the
Iitaka model of $(X, (m-1)K_{X/W}+f^*K_W)$ dominates $W$.
\end{proof}

Since $K_W$ is not necessarily big, we cannot directly apply Lemma
\ref{har}. But we still have:
\begin{lemm}\label{4.6}Under the assumptions of Lemma \ref{4.5}, the sheaf $$f_*(\cO_X(K_X+(m-1)K_{X/W}+f^*K_W)\otimes
 \cJ(||(m-1)K_{X/W}+f^*K_W||))$$ is nonzero, of rank
 $P_m(X_w)$, where $X_w$ is a general fiber of $f$.
\end{lemm}
\begin{proof}We use the diagram (\ref{d9}). Since $W_1$ is of general type, as in Lemma \ref{har}, by Viehweg's result, there exists $k>0$ such
that for $w_1$ a general point of $W_1$ and $X_{w_1}\subset X_1$ the
fiber of $f_1$, the restriction:
$$H^0(X_1, \cO_{X_1}(km(m-1)K_{X_1/W_1}+kmf_1^*K_{W_1}))\rightarrow H^0(X_{w_1}, \cO_{X_{w_1}}(km(m-1)K_{X_{w_1}}))$$
is surjective. Since $K_{W_1/W}\succeq 0$, by (\ref{d10}), we have
\begin{eqnarray*}&&H^0(X_1, \cO_{X_1}(km(m-1)K_{X_1/W_1}+kmf_1^*K_{W_1}))\\ &\subseteq & H^0(X_1, \cO_{X_1}(km(m-1)\pi_{X_1}^*K_{X/W}+km\pi_{X_1}^*f^*K_{W}+km(m-1)E)).\end{eqnarray*}
Since $E$ is $\pi_{X_1}$-exceptional, we conclude that
\begin{eqnarray*}&&|km(m-1)\pi_{X_1}^*K_{X/W}+km\pi_{X_1}^*f^*K_{W}+km(m-1)E|\\&=&
|km(m-1)\pi_{X_1}^*K_{X/W}+km\pi_{X_1}^*f^*K_{W}|+km(m-1)E.\end{eqnarray*}
We also know that $f_1(E)$ is a proper subvariety of $W_1$. These
imply that the restriction:
\begin{multline}\label{d11}H^0(X_1,
\cO_{X_1}(km(m-1)\pi_{X_1}^*K_{X/W}+km\pi_{X_1}^*f^*K_{W}))\\\rightarrow
H^0(X_{w_1}, \cO_{X_{w_1}}(km(m-1)K_{X_{w_1}}))\end{multline} is
surjective.

Set $w=b_{W_1}(w_1)$, and let $X_{w}$ be the fiber of $f$. In the
following diagram
$$
\xymatrix{ \pi_{X_1}^{-1}f^{-1}(U)\ar[r]\ar[d]&f^{-1}(U)\ar[d]\\
b_{W_1}^{-1}(U)\ar[r]&U,}$$ all the morphisms are smooth. Hence
$\pi_{X_{w_1}}=\pi_{X_1}|_{X_{w_1}}: X_{w_1}\rightarrow X_w$ is
\'{e}tale and the pull-back of
$H^0(X_w,\cO_{X_{w}}(km(m-1)K_{X_{w}}))$ is a subspace of
$H^0(X_{w_1},\cO_{X_{w_1}}(km(m-1)K_{X_{w_1}}))$.

On the other side, we have \begin{eqnarray}\label{d12}&&H^0(X_1,
\cO_{X_1}(k(m-1)\pi_{X_1}^*K_{X/W}+k\pi_{X_1}^*f^*K_{W}))\nonumber\\&=&\bigoplus_{\chi\in
G^*}\pi_{X_1}^*H^0(X, \cO_{X}(k(m-1)K_{X/W}+kf^*K_{W})\otimes
P_{\chi}).\end{eqnarray}Let $M$ be the order of $G$. Take a
resolution $\t: X^{'}\rightarrow X$ such that $\t:
X^{'}_w\rightarrow X_w$ is also a resolution and
\begin{itemize}
\item $\t^*|Mkm(m-1)K_{X/W}+Mkmf^*K_{W}|=|H|+E_M$,
\item $\t^*|\cO_{X}(km(m-1)K_{X/W}+kmf^*K_{W})\otimes
P_{\chi}|=|H_{\chi}|+E_{\chi}$, for each $\chi\in G^*$,
\item $\t^*|km(m-1)K_{X_{w}}|=|H_{w}|+E_{w}$,
\item $\t^*|mK_{X_w}|=|H_w^{'}|+E_w^{'}$,
\end{itemize}
such that $H$, $H_{\chi}$, $H_w$, and $H_w^{'}$ are base-point-free
and $E_M$, $E_{\chi}$, $E_w$, $E_w^{'}$ are the fixed divisors, with
SNC supports.

Let $X_1^{'}$ be a smooth model of the main component of
$X_{1}\times_{X}X^{'}$ (the irreducible component that dominates
$X_1$). We have the following commutative diagram:
\begin{eqnarray*}
\xymatrix{X_1^{'}\ar[rr]^{\pi_{X_1^{'}}}\ar[d]^{\tau_1}&& X^{'}\ar[d]^{\t}\\
X_1\ar[rr]^{\pi_{X_1}}\ar[d]^{f_1}&& X\ar[d]^f\\
W_1\ar[rr]^{b_{W_1}}&& W.}
\end{eqnarray*}
Let $U_1=X_1-E$. Then $\pi_{X_1}$ is \'{e}tale on $U_1$, hence
$U_1\times_{X}X^{'}$ is irreducible and smooth. Since $f_1(E)$ is a
proper subvariety of $W_1$, we can assume that there exists a
divisor $E^{'}$ of $X_1^{'}$ such that $X_1^{'}-E^{'}$ is just
$U_1\times_{X}X^{'}$ and $f_1\tau_1(E^{'})$ is a proper subvariety
of $W_1$. Let $X_{w_1}^{'}$ be the fiber of $f_1\tau_1$. Then
$\pi_{X_{w_1}^{'}}=\pi_{X_1^{'}}|_{X_{w_1}^{'}}:
X_{w_1}^{'}\rightarrow X_{w}^{'}$ is Galois \'{e}tale. We have
another commutative diagram involving   morphisms between fibers:
\begin{eqnarray*}
\xymatrix{ X_{w_1}^{'}\ar[rr]^{\pi_{X_{w_1}^{'}}}_{\rm\acute etale}
\ar[dd]^{\tau_1}_{1:1}
&&X_{w}^{'}\ar[dd]_{\tau}^{1:1}\\\\
X_{w_1}\ar[rr]^{\pi_{X_{w_1}}}_{\rm\acute etale} &&X_w. }
\end{eqnarray*}
We then write
\begin{eqnarray*}
&&\t_1^*|km(m-1)\pi_{X_1}^*K_{X/W}+km\pi_{X_1}^*f^*K_{W}|\\
&=&|\pi_{X_1}^{'*}\tau^*(km(m-1)K_{X/W}+kmf^*K_{W})|\\
&=&|H^{'}|+E_1^{'},
\end{eqnarray*}
where $E_1^{'}$ is the fixed divisor. Let $F$ be the maximal divisor which is
$\preceq E_{\chi}$ for all $\chi\in G^*$. By (\ref{d12}),
$\pi_{X_1^{'}}^*F\preceq E_1^{'}$. Hence, by (\ref{d11}), we conclude that
$\pi_{X_1^{'}}^*F|_{X_{w_1}^{'}}$ is fixed in
$\tau_1^*|km(m-1)K_{X_{w_1}}|$ and in particular is fixed in
$\pi_{X_{w_1}^{'}}^*\tau^*|km(m-1)K_{X_{w}}|$, so
$\pi_{X_1^{'}}^*F|_{X_{w_1}^{'}}\preceq \pi_{X_{w_1}^{'}}^*E_w$.
Since $\pi_{X_{w_1}^{'}}$ is \'{e}tale, we have
\begin{eqnarray*}\pi_{X_{w_1}^{'}}^*(F|_{X_w^{'}})\preceq \pi_{X_1^{'}}^*F|_{X_{w_1}^{'}}\preceq
\pi_{X_{w_1}^{'}}^*E_w.
\end{eqnarray*}

We conclude that $F|_{X_w^{'}}\preceq E_{w}$.

Since for any $\chi\in G^*$, we have the natural multiplication
\begin{multline*} H^0(X,
\cO_{X}(km(m-1)K_{X/W}+kmf^*K_{W})\otimes P_{\chi})^{\otimes
M}\\
\rightarrow H^0(X, \cO_{X}(Mkm(m-1)K_{X/W}+Mkmf^*K_{W})), \end{multline*}
we obtain $E_M\preceq MF$, hence $E_M|_{X^{'}_w}\preceq
ME_{w}\preceq Mk(m-1)E_w^{'}.$ This is just (\ref{(3)}) in the proof of 2)
of Lemma \ref{har}, and we can then finish the proof as there.
\end{proof}

 We may write Lemma \ref{4.6} in a more general form:

\begin{prop}Assume that we have the following commutative diagram between smooth projective varieties:
\begin{eqnarray*}
\xymatrix@C=40pt{X_1\ar[r]^{\pi_{X_1}}\ar[d]^{f_1}& X\ar[d]^{f}\\
W_1\ar[r]^{b_{W_1}} & W,}\end{eqnarray*} where $P_m(X)>0$, the morphism
$\pi_{X_1}$ is birationally equivalent to an \'{e}tale morphism and its
  exceptional divisor $E$ is such that $f_1(E)$ is a proper subvariety
of $W_1$, $\pi_{X_1*}\cO_{X_1}=\bigoplus_{\alpha}P_{\alpha}$ is a
direct sum of torsion line bundles on $X$, $W_1$ is of general type,
and $b_{W_1}$ is generically finite and surjective. Then the sheaf
$$f_*(\cO_X(K_X+(m-1)K_{X/W}+f^*K_W)\otimes
 \cJ(||(m-1)K_{X/W}+f^*K_W||))$$ is nonzero, of rank
 $P_m(X_w)$, where $X_w$ is a general fiber of $f$.
\end{prop}

According to Lemma \ref{4.6},
$$\cF_X=b_{W*}f_*(\cO_X(K_X+(m-1)K_{X/W}+f^*K_W)\otimes
 \cJ(||(m-1)K_{X/W}+f^*K_W||)))$$ is a nonzero sheaf on $A/K$.
 By Lemma \ref{mul} and Lemma \ref{4.5}, it is an IT-sheaf of index
0.

Let $\widehat{\cF_X}$ be the Fourier-Mukai transform of $\cF_X$. By the properties of this
  transformation (\cite{Mu}, Theorem
2.2), we know that $\widehat{\cF_X}$ is a W.I.T-sheaf of index $\dim
(A/K)$ and its Fourier-Mukai transform $\widehat{\widehat{\cF_X}}$ is
isomorphic to $(-1_{A/K})^*\cF_X$. In particular,
$\widehat{\cF_X}\neq 0$. Therefore, by the Base Change Theorem  and the definition of the Fourier-Mukai transform,
there exists $P_0\in \Pic^0(A/K)$ such that $h^0(A/K, \cF_X\otimes
P_0)\neq 0$. Thus for any $P\in \Pic^0(A/K)$,
$$h^0(A/K, \cF_X\otimes P)=\chi(\cF_X\otimes P)=\chi(\cF_X\otimes P_0)=h^0(A/K,
\cF_X\otimes P_0)\geq 1.$$

Hence for any $P\in\Pic^0(A/K)$, we have
\begin{eqnarray}\label{d13}& &h^0(X, \cO_X(K_X+(m-1)K_{X/W}+f^*K_W)\otimes f^*b_W^*P)\nonumber\\&\geq& h^0\big(X, \cO_X(K_X+(m-1)K_{X/W}+f^*K_W)\otimes
\nonumber \\&&\qquad\qquad\qquad\qquad\cJ(||(m-1)K_{X/W}+f^*K_W||)\otimes f^*b_W^*P\big)\nonumber\\&=&h^0\big(A/K, b_{W*}f_*(\cO_X(K_X+(m-1)K_{X/W}+f^*K_W)\otimes
\nonumber\\&& \qquad\qquad\qquad\qquad\cJ(||(m-1)K_{X/W}+f^*K_W||))\otimes P\big)\nonumber\\&=&h^0(A/K, \cF_X\otimes P)\nonumber\\&\geq& 1.
 \end{eqnarray}
\begin{lemm}\label{4.7}Let $X$ and $W$ be as in Lemma \ref{4.5}. Suppose   $\kappa(W)>0$. Then for any $r\geq 3$, there exists
a translate $T\subset\Pic^0(A/K)$ of a positive-dimensional torus,
such that
$$h^0(W, \cO_W((r-2)K_W)\otimes b_W^*P)\geq r-2,$$ for all $P\in T$.
\end{lemm}
\begin{proof}Since $\kappa(W)>0$, there exist a positive-dimensional abelian subvariety $T_0\subset \Pic^0(A/K)$ and a torsion point $P_0\in\Pic^0(A/K)$ such that $b_W^*(P_0+T_0)\subset V_0(\omega_W)$ (\cite{CH2}, Corollary 2.4).
Then we iterate Lemma \ref{sim} to get $h^0(W,
\cO_W((r-2)K_W)\otimes b_W^*P)\geq r-2,$ for all $P\in
(r-2)P_0+T_0$.
\end{proof}

If $\kappa(W)>0$, since $mK_X=K_X+(m-1)K_{X/W}+f^*K_W +
(m-2)f^*K_W$, again by (\ref{d13}), Lemma \ref{4.7} and Lemma \ref{sim}, we obtain
$$P_m(X)\geq 1+m-2+\dim (T)-1\geq m-1,$$
which contradicts our assumption. Hence we have finished the proof in the case $\kappa(W)>0$.

If $\kappa(W)=0$, in the diagram (\ref{d8}), $b_W$ is surjective and finite
and $\kappa(W)=0$, hence $W$ is an abelian variety by Kawamata's
Theorem \ref{K2}. We still have (\ref{d13}), however $K_W$ is trivial,
hence it is not enough for us to deduce a contradiction. We will
need new versions of Lemma \ref{4.5} and Lemma \ref{4.6}.

First we go back to diagrams (\ref{d8}) and (\ref{d9}):
\begin{eqnarray*}
\xymatrix@C=40pt{
   X_1 \ar[r]^{\pi_{X_1}}\ar[d]^{g_1}\ar@/_2pc/[dd]_{f_1} & X\ar[d]^{g}\ar@/^2pc/[dd]^{f} \\
   V_1 \ar[r]^{\pi_{V_1}}\ar[d]^{h_1}& V\ar[d]^{h} \\
   W_1 \ar[r]^{b_{W_1}}& W,
}
\end{eqnarray*}
where $V_1$ is birational to $\KK\times W_1$.

Since $\pi_{V_1}: V_1\rightarrow V$ is birationally equivalent to
the \'{e}tale cover $\VV\rightarrow V$, we have
$\pi_{V_1*}\omega_{V_1}=\bigoplus_{\chi\in G^*}(\omega_V\otimes
P_{\chi})$. On the other hand, $V_1$ is birational to $\KK\times
W_1$, hence $h_{1*}\omega_{V_1}=\omega_{W_1}$. Therefore, we have
$$b_{W_1*}\omega_{W_1}=\bigoplus_{\chi\in G^*}h_*(\omega_V\otimes
P_{\chi}).$$

Since $b_{W_1}$ is generically finite and $W_1$ is of general type, by
Theorem 2.3 in \cite{CH2}, we know that the irreducible components
of $V_0(b_{W_1*}\omega_{W_1})$ generate $\Pic^0(W)$. Hence there
exists a $\chi\in G^*$ such that $V_0(h_*(\omega_V\otimes
P_{\chi}))$ is a translated positive-dimensional abelian subvariety
of $\Pic^0(W)$. We denote $h_*(\omega_V\otimes P_{\chi})$ by
$\cF_{\chi}$. Since a general fiber of $h$ is an abelian variety,
$\cF_{\chi}$ is a rank-$1$ torsion-free sheaf.

We can again birationally modify $X$ so that $f^*\cF_{\chi}$ is a
line bundle on $X$. We then have the following result similar to
Lemma \ref{4.5}.
\begin{lemm}\label{L1}Under the assumptions of Lemma \ref{4.5}, assume moreover that $\kappa(W)=0$ and let $\cF_{\chi}$ be as
above. Then the Iitaka model of $(X, (m-1)K_{X}-(m-2)f^*\cF_{\chi})$
dominates $W$.
\end{lemm}
\begin{proof}The proof is analogue to that of Lemma \ref{4.5}. We
have
\begin{eqnarray}\label{11*}
 &&\pi_{X_1}^*((m-1)K_{X}-(m-2)f^*\cF_{\chi})+(m-1)E
\\&=&(m-1)K_{X_1/W_1}+f_1^*K_{W_1}+(m-2)f_1^*K_{W_1}-(m-2)\pi_{X_1}^*f^*\cF_{\chi}
.\nonumber
\end{eqnarray}
Since $\cF_{\chi}\subset b_{W_1*}\omega_{W_1}$, we have an inclusion $b_{W_1}^*\cF_{\chi}\hookrightarrow \omega_{W_1}$, hence an inclusion
$$(m-2)f_1^*b_{W_1}^*\cF_{\chi}=(m-2)\pi_{X_1}^*f^*\cF_{\chi}\hookrightarrow(m-2)f_1^*\omega_{W_1}.$$ Using Viehweg's result
as in the proof of Lemma \ref{4.5}, we obtain that the Iitaka model of
$\pi_{X_1}^*((m-1)K_{X}-(m-2)f^*\cF_{\chi})+(m-1)E$ dominates $W_1$.
We   finish the proof by the same argument as in Lemma \ref{4.5}.
\end{proof}

We also need an analogue of Lemma \ref{4.6}.

\begin{lemm}\label{L2}Under the same assumptions as in Lemma \ref{L1}, the sheaf
$$f_*(\cO_X(mK_X-(m-2)f^*\cF_{\chi})\otimes
\cJ(||(m-1)K_{X}-(m-2)f^*\cF_{\chi}||))$$ is nonzero of
rank $P_m(X_w)$, where $X_w$ is a general fiber of $f$.
\end{lemm}
\begin{proof}It is also parallel to the proof of Lemma \ref{4.6}.
First, by Viehweg's result again, we have the surjectivity of the
restriction map:
\begin{multline*}
H^0(X_1, \cO_{X_1}(km(m-1)K_{X_1/W_1}+kmf_1^*K_{W_1}))\\
\rightarrow  H^0(X_{w_1}, \cO_{X_{w_1}}(km(m-1)K_{X_{w_1}})).
\end{multline*}
Since $E$ is $\pi_{X_1}$-exceptional and $(m-2)f_1^*K_{W_1}\succeq
(m-2)\pi_{X_1}^*f^*\cF_{\chi}$, by (\ref{11*}), we have the surjectivity
of the restriction map:
\begin{multline*}
H^0(X_1, \pi_{X_1}^*\cO_{X_1}(km(m-1)K_{X}-km(m-2)f^*\cF_{\chi}))\\
\rightarrow H^0(X_{w_1}, \cO_{X_{w_1}}(km(m-1)K_{X_{w_1}})).
\end{multline*}
Then the rest of the proof is the same as the proof of Lemma
\ref{4.6}.
\end{proof}

By Lemma \ref{L1} and Lemma \ref{L2}, we again conclude as in (\ref{d13}) that
$$h^0(X, \cO_X(mK_{X}-(m-2)f^*\cF_{\chi})\otimes f^*P)\geq 1,$$ for
any $P\in\Pic^0(W)$.

As in the proof of Lemma \ref{4.7}, there exists a translate $T\subset
\Pic^0(W)$ of a positive-dimensional abelian variety  such that
$h^0(X, \cO_X((m-2)f^*\cF_{\chi})\otimes f^*P)\geq m-2$, for any
$P\in T$. We again have $P_m(X)\geq m-1$, which is a contradiction.
This finishes the proof of Theorem \ref{10} in the case $\kappa(W)=0$.

In all, we have finished the proof of Theorem \ref{10}.
\end{proof}

 \section*{Acknowledgements}
I am extremely grateful to my thesis advisor, O. Debarre, for
his patient corrections and helpful remarks. I would also like to thank J.A. Chen and
C.D. Hacon for several helpful conversations.

\end{document}